\documentclass[12pt]{article}
\usepackage{amssymb,amsfonts,amsmath,amsrefs,amsthm,hyperref,bbm,dsfont,tikz,wrapfig,tikz-cd}
\usepackage{geometry}
\newcommand{\8}{\infty}

\newcommand{\R}{\mathbb{R}}
\newcommand{\C}{\mathbb{C}}
\newcommand{\N}{\mathbb{N}}

\newcommand{\D}{\mathbb{D}}

\newcommand{\dil}{\mathrm{dil}}

\newcommand{\spa}{\mathrm{span}}
\newcommand{\spac}{\overline{\mathrm{span}}}
\newcommand{\Ker}{\mathrm{Ker~}}

\newcommand{\supp}{\mathrm{supp}}
\newcommand{\Co}{\mathcal{C}}
\newcommand{\Fo}{\mathcal{F}}

\newcommand{\Po}{\mathcal{P}}

\newcommand{\Lo}{\mathcal{L}}

\newcommand{\Ho}{\mathcal{H}}

\newcommand{\B}{\mathrm{B}}
\newcommand{\Bo}{\overline{\B}}
\newcommand{\1}{\mathds{1}}

\newcounter{erz}[section] \numberwithin{erz}{section}
\newtheorem{theorem}[erz]{Theorem}
\newtheorem{lemma}[erz]{Lemma}
\newtheorem{proposition}[erz]{Proposition}

\newtheorem{corollary}[erz]{Corollary}
\newtheorem{question}[erz]{Question}
\theoremstyle{remark}
\newtheorem{remark}[erz]{Remark}
\newtheorem{example}[erz]{Example}

\begin{document}
\title{Multiplier algebras of normed spaces of continuous functions}
\author{Eugene Bilokopytov\footnote{Email address bilokopi@myumanitoba.ca, erz888@gmail.com.}}
\maketitle

\begin{abstract}
In this article we investigate some general properties of the multiplier algebras of normed spaces of continuous functions (NSCF). In particular, we prove that the multiplier algebra inherits some of the properties of the NSCF. We show that it is often possible to construct NSCF's which only admit constant multipliers. In order to do that, using a method from \cite{mr}, we prove that any separable Banach space can be realized as a NSCF over any separable metrizable space. On the other hand, we give a sufficient condition for non-separability of a multiplier algebra.\medskip


MSC2020 46E15, 46B28, 47B38
\end{abstract}

\section{Introduction}

Multiplier algebras constitute the basic class of examples of operator algebras, and gained a lot of attention recently. Mostly, such algebras are considered for Hilbert spaces of holomorphic functions, or Hilbert spaces of functions on sets without any further structure. In this paper we try to widen the scope of the subject by exploring some basic properties of the multiplier algebras of normed spaces of continuous functions.

First, let us define precisely what we mean by a normed space of continuous functions. Let $X$ be a topological space (a \emph{phase space}) and let $\Co\left(X\right)$ denote the space of all continuous complex-valued functions over $X$ endowed with the compact-open topology. A \emph{normed space of continuous functions} (NSCF) over $X$ is a linear subspace $\mathbf{F}\subset\Co\left(X\right)$ equipped with a norm that induces a topology, which is stronger than the compact-open topology, i.e. the inclusion operator $J_{\mathbf{F}}:\mathbf{F}\to\Co\left(X\right)$ is continuous, or equivalently the unit ball $\B_{\mathbf{F}}$ is bounded in $\Co\left(X\right)$. If $\mathbf{F}$ is a linear subspace of $\Co\left(X\right)$, then the \emph{point evaluation} at $x\in X$ on $\mathbf{F}$ is the linear functional $x_{\mathbf{F}}:\mathbf{F}\to\C$, defined by $x_{\mathbf{F}}\left(f\right)=f\left(x\right)$. If $\mathbf{F}$ is a NSCF, then all point evaluations are bounded on $\mathbf{F}$. Conversely, if $\mathbf{F}\subset\Co\left(X\right)$ is equipped with a complete norm such that $x_{\mathbf{F}}\in \mathbf{F}^{*}$, for every $x\in X$, then $\mathbf{F}$ is a NSCF. We will call a NSCF $\mathbf{F}$ over $X$ \emph{(weakly) compactly embedded} if $J_{\mathbf{F}}$ is a (weakly) compact operator, or equivalently, if $\B_{\mathbf{F}}$ is (weakly) relatively compact in $\Co\left(X\right)$. Clearly, every compactly embedded NSCF's is weakly compactly embedded. On the other hand, any reflexive NSCF is also weakly compactly embedded. By a Banach / Hilbert space of continuous functions (BSCF / HSCF) we mean a complete / Hilbert NSCF.

A \emph{multiplication operator} (MO) with \emph{symbol} $\omega:X\to\C $ is a linear map $M_{\omega}$ on the space $\Fo\left(X\right)$ of all complex-valued functions on $X$ defined by $$\left[M_{\omega}f\right]\left(x\right)=\omega\left(x\right)f\left(x\right),$$ for $x\in X$. Let $\mathbf{F}$ and $\mathbf{E}$ be NSCF's over $X$. If $M_{\omega}\mathbf{F}\subset\mathbf{E} $ and $\left.M_{\omega}\right|_{\mathbf{F}}\in\Lo\left(\mathbf{F},\mathbf{E}\right)$, then we say that $M_{\omega}$ is a multiplication operator from $\mathbf{F}$ into $\mathbf{E}$. If in this case $\mathbf{F}=\mathbf{E}$, then we will call $\omega$ a \emph{multiplier} of $\mathbf{F}$.\medskip

The collection of multipliers of a NSCF constitutes a NSCF of its own, which additionally is an algebra. In this article we investigate some general properties of this object. In particular, we prove that the multiplier algebra inherits some of the properties of the NSCF (see Proposition \ref{muw}, Theorem \ref{mus} and Proposition \ref{muw2}). One of the features of the multiplier algebras is that it is difficult to predict how ``large'' they are. We show that it is often possible to construct NSCF's which only admit constant multipliers (see Proposition \ref{anm} and Corollary \ref{anm2}). In order to do that, using a method from \cite{mr}, we prove that any separable Banach space can be realized as a NSCF over any separable metrizable space (see Theorem \ref{real}). On the other hand, we give a sufficient condition for non-separability of a multiplier algebra (see Theorem \ref{sep}).

Let us describe the structure of the article. In Section 2 we study the multiplier algebras of NSCF's independent of the topology of their phase spaces. In Section 3 we recall some basic facts about NSCF's and show that it is a very broad category. In Section 4 we bring the topology of the phase space back into the mix, and it enables us to construct various examples of NSCF's with no non-constant multipliers. Finally, in Section 5 we consider subalgebras generated by some finite collections of multipliers.\bigskip

Throughout the paper by $Id_{X}$ we mean the identity map on $X$, and $\1$ is the constant function $1$. For $Y\subset X$ the supremum semi-norm of $f:X\to\C$ is denoted by $\|f\|_{\8}^{Y}$; if $Y=X$ we simply use $\|f\|_{\8}$.

Before concluding this section with some concrete examples of NSCF's, let us mention a large class of compactly embedded NSCF's. If $X$ is a domain in $\C^{n}$, i.e. an open connected set, and $\mathbf{F}$ is a NSCF over $X$ that consists of holomorphic functions, then $\mathbf{F}$ is compactly embedded. Indeed, by Montel's theorem (see \cite[Theorem 1.4.31]{scheidemann}), $B_{\mathbf{F}}$ is relatively compact in $\Co\left(X\right)$, since it is a bounded set that consists of holomorphic functions. We will call such NSCF's \emph{normed spaces of holomorphic functions} (NSHF). The notions of BSHF and HSHF are defined analogously.

\begin{example}\label{wuh}
Assume that $X$ is a domain in $\C^{n}$ and let $u:X\to\left(0,+\8\right)$ be continuous. Define the \emph{weighted space of holomorphic functions}  $$\Ho_{u}^{\8}=\left\{f\in\Ho\left(X\right),~\|f\|_{u}^{\8}=\|uf\|_{\8}<+\8\right\}.$$ One can show that this is a BSHF over $X$ with respect to the norm $\|\cdot\|_{u}^{\8}$.  If $u= \1$ we will use the notation $\Ho_{\8}\left(X\right)$.

Assume additionally that $X$ is bounded. Let $A\left(X\right)$ be the closed subalgebra of $\Ho_{\8}\left(X\right)$ which consists of functions that admit a continuous extension on $\overline{X}$. Under some assumptions about $X$ (see e.g. \cite[Theorem 2.1]{range} and \cite{gn}) $A\left(X\right)$ is the closure of the polynomials with respect to $\|\cdot\|_{\8}$. Another natural way to represent this space is a NSCF over $\overline{X}$ that consists of functions holomorphic on $X$.\qed\end{example}\smallskip

\begin{example}\label{hardy}
For $p\in \left[1,+\8\right]$ the \emph{Hardy space} $\Ho^{p}$ is a BSHF over the (open) unit disk $\D\subset\C$ that consists of holomorphic functions $f$ with the norm defined by $$\|f\|^{p}=\sup\limits_{r\in\left[0,1\right)}\int\limits_{\partial \D}\left|f\left(re^{i\theta}\right)\right|^{p}d\theta.$$ One can show that if $p=2$ this is a Hilbert space; when $p=+\8$, we get $\Ho^{\8}=\Ho_{\8}\left(\D\right)$. It also follows from Holder's inequality that $\Ho^{p}\subset \Ho^{q}$, when $q\le p$. The Hardy space is among the most studied function spaces, and we refer to e.g. \cite{duren} for more information.
\qed\end{example}

Several additional examples of NSCF's will be introduced throughout the paper.

\section{Multipliers of a NSF}

In this section we perform the initial study of the multiplier algebra of a NSCF, but with no regard to the topology of its phase space. In order to do that we need to adjust the definition of NSCF. Everywhere in this section $X$ is a set. A \emph{normed space of functions} (NSF) over $X$ is a NSCF over $X$ endowed with the discrete topology. The notions of BSF and HSF are defined analogously. Clearly, if $\mathbf{F}$ is a NSCF over a topological space $Y$, it is a NSF over the set $Y$. Let us consider an important subclass of NSF's.

\begin{example}\label{rkhs}
For a HSF $\mathbf{H}$ over $X$ consider a function $K_{\mathbf{H}}:X\times X\to\C$ defined by $K_{\mathbf{H}}\left(x,y\right)=\left<x_{\mathbf{H}},y_{\mathbf{H}}\right>$. The principal property of $K_{\mathbf{H}}$ is that it is a \emph{(positive semi-definite) kernel}, i.e. for every $x_{1},...,x_{n}\in X$ the matrix $\left[K_{\mathbf{H}}\left(x_{i},x_{j}\right)\right]_{i,j=1}^{n}$ is positive semi-definite. Conversely, if $K:X\times X\to\C$ is a kernel, there is a unique HSF $\mathbf{H}_{K}$ over $X$ such that $K_{\mathbf{H}_{K}}=K$ (see \cite[Theorem 2.23]{am}). Because of this, the traditional term for what we call HSF is \emph{Reproducing Kernel Hilbert Space}.

It is easy to see that if $\omega:X\to\C$, then $\omega\otimes\overline{\omega}:X\times X\to\C$ defined by $\omega\otimes\overline{\omega}\left(x,y\right)=\omega\left(x\right)\overline{\omega\left(y\right)}$ is a kernel. One can show that kernels form a closed convex cone in $\Fo\left(X\times X\right)$. Since from Schur's product theorem an entrywise product of positive semi-definite matrices is a positive semi-definite matrix, it follows that a product of kernels is a kernel. In particular, if $K:X\times X\to\D$ is a kernel, then $L=\frac{1}{1-K}=\sum\limits_{n=0}^{+\8} K^{n}$ is also a kernel.  For example, the kernel of the Hardy space is the Szego kernel $K_{\Ho^{2}}\left(z,w\right)=\frac{1}{1-z\overline{w}}$  (see \cite[Example 2.9]{am}).
\qed\end{example}\medskip

Let us turn to the characterizations and some basic properties of the multiplication operators. The following is an immediate consequence of the Closed Graph Theorem.

\begin{proposition}\label{clogr}
Let $\mathbf{F}$ and $\mathbf{E}$ be BSF's over $X$ and let $\omega:X\to\C $ be such that $M_{\omega}\mathbf{F}\subset\mathbf{E}$. Then $M_{\omega}\in\Lo\left(\mathbf{F},\mathbf{E}\right)$. In particular, if $\mathbf{F}\subset\mathbf{E}$, then the inclusion operator is continuous, and if $\mathbf{F}=\mathbf{E}$ as sets, their norms are equivalent.
\end{proposition}

For $Y\subset X$ define $\mathbf{F}_{Y}=\left\{f\in \mathbf{F}\left|\supp f\subset Y\right.\right\}=\left\{x_{\mathbf{F}}, x\in X\backslash Y\right\}^{\perp}$, which is a closed subspace of $\mathbf{F}$, and is also a NSF over $X$.

\begin{remark}\label{clogr2}
The first claim of the proposition still holds if $\mathbf{E}$ has a semi-norm such that there is $Y\subset X$ with $\mathbf{E}=\mathbf{E}_{Y}+\mathbf{E}_{X\backslash Y}$ and $\Ker \|\cdot\|=\mathbf{E}_{X\backslash Y}$.
\qed\end{remark}

Observe that $\Ker M_{\omega}=\left\{x_{\mathbf{F}}\left| \omega\left(x\right)\ne 0\right.\right\}^{\perp}=\mathbf{F}_{\omega^{-1}\left(0\right)}$ and $\overline{M_{\omega}\mathbf{F}}\subset\mathbf{E}_{X\backslash\omega^{-1}\left(0\right)}$. Related to this is the following characterization of MO's (cf. \cite{barb}).

\begin{proposition}\label{char}
Let $\mathbf{F}$ and $\mathbf{E}$ be NSF's over $X$ and let $T\in\Lo\left(\mathbf{F},\mathbf{E}\right)$. The following are equivalent:
\item[(i)] There is $\omega:X\to\C$ such that $T=M_{\omega}$;
\item[(ii)] $T\mathbf{F}_{Y}\subset \mathbf{E}_{Y}$, for every $Y\subset X$;
\item[(iii)] $T\mathbf{F}_{X\backslash \left\{x\right\}}\subset \mathbf{E}_{X\backslash \left\{x\right\}}$, for every $x\in X$;
\item[(iv)] $T^{*}x_{\mathbf{E}}\in \C x_{\mathbf{F}} $, for every $x\in X$.
\end{proposition}
\begin{proof}
The implications (i)$\Rightarrow$(ii)$\Rightarrow$(iii) are obvious, and (iv)$\Rightarrow$(i) is standard (see e.g. \cite[Proposition 2.4 and Corollary 2.5]{erz1}). Let us prove (iii)$\Rightarrow$(iv). Let $x\in X$. Since $\mathbf{F}_{X\backslash \left\{x\right\}}=\left\{x_{\mathbf{F}}\right\}^{\bot}$ in $\mathbf{F}$, we have $T\left\{x_{\mathbf{F}}\right\}^{\bot}\subset \left\{x_{\mathbf{E}}\right\}^{\bot}$, from where $T^{*}x_{\mathbf{E}}\in T^{*}\left\{x_{\mathbf{E}}\right\}^{\bot\bot}\subset \left\{x_{\mathbf{F}}\right\}^{\bot\bot}=\C x_{\mathbf{F}} $.
\end{proof}

For NSF's $\mathbf{F}$ and $\mathbf{E}$ over $X$ let $Mult\left(\mathbf{F}, \mathbf{E}\right)$ be the collection of all $\omega:X\to\C$ such that $M_{\omega}\in\Lo\left(\mathbf{F},\mathbf{E}\right)$.  Clearly, $Mult\left(\mathbf{F}, \mathbf{E}\right)$ is a linear subspace of $\Fo\left(X\right)$ and $\|\cdot\|_{Mult\left(\mathbf{F},\mathbf{E}\right)}$ defined by $\|\omega\|_{Mult\left(\mathbf{F}, \mathbf{E}\right)}=\|M_{\omega}\|$ is a seminorm. Note that the topology of $Mult\left(\mathbf{F},\mathbf{E}\right)$ depends only on the topology of $\mathbf{F}$ and $\mathbf{E}$. Hence, re-norming the latter results in an equivalent norm on the former.

If $\mathbf{F}$ and $\mathbf{E}$ are complete, then according to Proposition \ref{clogr}, $Mult\left(\mathbf{F},\mathbf{E}\right)$ is the collection of all $\omega$ such that $M_{\omega}\mathbf{F}\subset \mathbf{E}$. Since a continuous linear operator between normed spaces can be extended to an operator between their completions with the same norm, it follows that $Mult\left(\mathbf{F},\mathbf{E}\right)$ isometrically embeds into $Mult\left(\overline{\mathbf{F}},\overline{\mathbf{E}}\right)$ in the case when $\overline{\mathbf{F}}$ and $\overline{\mathbf{E}}$ are NSF's.

Note that in general the equality of MO's as operators does not imply the equality of their symbols. We will call a NSCF $\mathbf{F}$ over $X$ $1$-\emph{independent} if for every $x\in X$ we have $x_{\mathbf{F}}\ne 0_{\mathbf{F}^{*}}$, i.e. there is $f\in \mathbf{F}$ such that $f\left(x\right)\ne 0$. It is easy to see that a MO from a $1$-independent NSCF determines its symbol, and $\|\cdot\|_{Mult\left(\mathbf{F},\mathbf{E}\right)}$ is a norm. Also, $\1\in Mult\left(\mathbf{F},\mathbf{E}\right)$ if and only if $\mathbf{F}\subset\mathbf{E}$, with $\|\1\|_{Mult\left(\mathbf{F},\mathbf{E}\right)}$ being the norm of the inclusion.

Since from Proposition \ref{char} we have that the set of MO's can be characterized as $\bigcap\limits_{x\in X}\left\{T\in \Lo\left(\mathbf{F},\mathbf{E}\right), \left<Tf,x_{\mathbf{E}}\right>=0, f\in \mathbf{F}_{X\backslash \left\{x\right\}}\right\}$, we get the following property.

\begin{corollary}\label{wot}
If $\mathbf{F}$ and $\mathbf{E}$ are NSF's over $X$, then $Mult\left(\mathbf{F}, \mathbf{E}\right)$ embeds as a closed subspace of $\Lo\left(\mathbf{F},\mathbf{E}\right)$ with respect to the weak operator topology. In particular, if $\mathbf{E}$ is complete, $Mult\left(\mathbf{F}, \mathbf{E}\right)$ is a Banach space.
\end{corollary}

In what follows we will view $Mult\left(\mathbf{F}, \mathbf{E}\right)$ as simultaneously a set of functions and a set of operators as long as it does not cause a confusion.

Since $Mult\left(\mathbf{F}, \mathbf{E}\right)$ is a linear subspace of $\Fo\left(X\right)$ with a semi-norm, it is natural to ask whether it is a NSF, and what properties it might have. It turns out that under the ``minimal'' assumption that $\mathbf{F}$ is $1$-independent, $Mult\left(\mathbf{F}, \mathbf{E}\right)$ inherits the main properties of $\mathbf{E}$.

\begin{proposition}\label{muw}Let $\mathbf{F}$ and $\mathbf{E}$ be NSF's over $X$. If $\mathbf{F}$ is $1$-independent, then:
\item[(i)] $Mult\left(\mathbf{F}, \mathbf{E}\right)$ is a NSF over $X$, and moreover the weak operator topology is stronger than the pointwise topology on $Mult\left(\mathbf{F}, \mathbf{E}\right)$.
\item[(ii)] If the pointwise topology coincides with the weak topology on $\Bo_{\mathbf{E}}$, then the pointwise topology coincides with the weak operator topology on $\Bo_{Mult\left(\mathbf{F}, \mathbf{E}\right)}$.
\end{proposition}
\begin{proof}
(i): Let $x\in X$ and let $f\in \mathbf{F}$ be such that $f\left(x\right)=1$. Then the semi-norm $|||\cdot|||$ on $Mult\left(\mathbf{F}, \mathbf{E}\right)$ defined by $|||\omega|||=\left|\left<M_{\omega}f,x_{\mathbf{E}}\right>\right|=\left|\omega\left(x\right)f\left(x\right)\right|=\left|\omega\left(x\right)\right|$ is continuous with respect to the weak operator topology. Sine $x$ was chosen arbitrarily, the second claim follows. Since the norm topology is stronger than the weak operator topology, we conclude that $Mult\left(\mathbf{F}, \mathbf{E}\right)$ is a NSF.\medskip

(ii): Assume that a net $\left\{\omega_{i}\right\}_{i\in I}\subset \Bo_{Mult\left(\mathbf{F}, \mathbf{E}\right)}$ converges pointwise to $0$. Then, for every $f\in \mathbf{F}$ and $i\in I$ we have $\|\omega_{i}f\|\le \|f\|$, and so the net $\left\{\omega_{i}f\right\}_{i\in I}$ is bounded in $\mathbf{E}$ and pointwise convergent to $0$. From the assumption about $\mathbf{E}$ it follows that $\omega_{i}f\xrightarrow{i\in I}0$ weakly. Hence, the pointwise topology is stronger than the weak operator topology on $\Bo_{Mult\left(\mathbf{F}, \mathbf{E}\right)}$. Combining this with (i) shows that these two topologies coincide on $\Bo_{Mult\left(\mathbf{F}, \mathbf{E}\right)}$.
\end{proof}

\begin{remark}
In a similar way as in part (i), one can show that if $\1\in\mathbf{F}$ then $Mult\left(\mathbf{F},\mathbf{E}\right)\subset \mathbf{E}$ with the strong operator topology being stronger than the topology of $\mathbf{E}$. Furthermore, the norm of the inclusion (with respect to the norm on $Mult\left(\mathbf{F},\mathbf{E}\right)$) is at most $\|\1\|_{\mathbf{F}}$. Indeed, if $\omega\in Mult\left(\mathbf{F},\mathbf{E}\right)$, then $\omega=M_{\omega} \1\in \mathbf{E}$ with $\|\omega\|\le \|\1\|_{\mathbf{F}}\|M_{\omega}\|=\|\1\|_{\mathbf{F}}\|\omega\|_{Mult\left(\mathbf{F},\mathbf{E}\right)}$.
\qed\end{remark}

\begin{remark}\label{sde}
Note that since $\Fo\left(X\right)$ is a reflexive locally convex space, there is $J_{\mathbf{E}}^{**}:\mathbf{E}^{**}\to \Fo\left(X\right)$. It was shown in \cite{erz3} that $J_{\mathbf{E}}^{**}$ is injective if and only if $\spac \left\{x_{\mathbf{E}}\left|x\in X\right.\right\}=\mathbf{E}^{*}$ and if and only if the weak and pointwise topologies coincide on $\B_{\mathbf{E}}$. In this case $\mathbf{E}^{**}$ is a BSF over $X$ and $J_{\mathbf{E}^{**}}=J_{\mathbf{E}}^{**}$. This happens in particular, when $\mathbf{E}$ is reflexive. There is a ``sequential'' variant of the statement. An example of a NS(C)F to which it is applicable is $\Co_{0}\left(X\right)$ (this follows from \cite[IV.6, Corollary 4]{ds}, applied to the one point compactification of $X$).
\qed\end{remark}\medskip

For a NSF $\mathbf{F}$ over $X$ let $Mult\left(\mathbf{F}\right)=Mult\left(\mathbf{F},\mathbf{F}\right)$. Clearly, $Mult\left(\mathbf{F}\right)$ is a unital algebra, $\|\cdot\|_{Mult\left(\mathbf{F}\right)}$ is a submultiplicative seminorm, and in particular, $\|\1\|_{Mult\left(\mathbf{F}\right)}=1$. Moreover, $\1\in \mathbf{F}$ if and only if $Mult\left(\mathbf{F}\right)\subset \mathbf{F}$, with the norm of the inclusion equal to $\|\1\|$, since $\|\1\|=\|\1\|\cdot\|\1\|_{Mult\left(\mathbf{F}\right)}$. It follows from Dedekind's independence theorem, that any collection of non-zero point evaluations is linearly independent on $Mult\left(\mathbf{F}\right)$, as long as the latter separates the corresponding points. If $\mathbf{F}$ is complete, then $Mult\left(\mathbf{F}\right)$ is a unital Banach Algebra, according to Corollary \ref{wot}. If $\mathbf{F}$ is a $1$-independent NSF over $X$, then $Mult\left(\mathbf{F}\right)$ contractively embeds into $\Fo_{\8}\left(X\right)$ (see \cite[Proposition 2.2]{erz3}).

\begin{proposition}\label{subn}
Let $\mathbf{F}$ be a BSF over $X$, which is a subalgebra of $\Fo\left(X\right)$. Then there is $\alpha>0$ such that $\|fg\|\le\alpha\|f\|\|g\|$, for any $f,g\in\mathbf{F}$. Moreover, the new norm $|||\cdot|||$ defined by $|||f|||=\alpha\|f\|$, for $f\in \mathbf{F}$, is submultiplicative.
\end{proposition}
\begin{proof}
If $\mathbf{F}$ is an algebra, then $M_{f}\mathbf{F}\subset \mathbf{F}$, for every $f\in \mathbf{F}$, and since it is a BSF, $\mathbf{F}\subset Mult\left(\mathbf{F}\right)$, due to Proposition \ref{clogr}. Moreover, from Remark \ref{clogr2} this inclusion is in fact continuous. Hence, there is $\alpha\ge 1$ such that $\|f\|_{Mult\left(\mathbf{F}\right)}\le\alpha\|f\|$, for all $f\in \mathbf{F}$, from where $\|fg\|\le \|g\|\|f\|_{Mult\left(\mathbf{F}\right)}\le\alpha\|f\|\|g\|$, $f,g\in \mathbf{F}$. The second claim is easy to verify.
\end{proof}

Let $K$ be a kernel on $X$ (see Example \ref{rkhs}), and let $\mathbf{H}_{K}$ be the corresponding HSF. One can show (see \cite[Corollary 2.37]{am}) that $\omega\in \Bo_{Mult\left(\mathbf{H}_{K}\right)}$ if and only if $\left(1-\omega\otimes\overline{\omega}\right)K$ is a kernel. Using this fact let us compare multiplier algebras of different HSF's.

\begin{proposition}If $K$ and $L$ are two kernels on $X$, then $Mult\left(\mathbf{H}_{K}\right)$ is contractively embedded into $Mult\left(\mathbf{H}_{KL}\right)$.
\end{proposition}
\begin{proof}
If $\omega\in\Bo_{Mult\left(\mathbf{H}_{K}\right)}$ then $\left(1-\omega\otimes\overline{\omega}\right)K$ is a kernel. Since the product of kernels is a kernel, it follows that $\left(1-\omega\otimes\overline{\omega}\right)KL$ is a kernel, from where $\omega\in\Bo_{Mult\left(\mathbf{H}_{KL}\right)}$.
\end{proof}

\begin{remark}
One can also show that $M_{L\left(\cdot,z\right)}$ is an operator from $\mathbf{H}_{K}$ into $\mathbf{H}_{KL}$, of norm $\sqrt{L\left(z,z\right)}$, for any $z\in X$.\qed
\end{remark}

\section{Every separable Banach space is a NSCF}

We interject the discussion of multiplier algebras to show that every separable Banach space can be realized as a BSCF over every separable metrizable space, which generalizes a result from \cite{mr}. Before doing that however, let us discuss some basic properties of NSCF's.

Until the end of the section $X$ is a Hausdorff topological space. We will often need to put certain restrictions on the phase spaces of NSCF's. Namely, $X$ is called \emph{compactly generated}, or a \emph{k-space} whenever each set which has closed intersections with all compact subsets of $X$ is closed itself. It is easy to see that all metrizable and all locally compact Hausdorff spaces are compactly generated. Moreover, Arzela-Ascoli theorem describes the compact subsets of $\Co\left(X\right)$ in the event when $X$ is compactly generated, which further justifies the importance of this class of topological spaces. Furthermore, if $X$ is compactly generated, then $\Co\left(X\right)$ is a complete locally convex space (see\cite[8.3.C]{engelking}). Additional details concerning the compactly generated spaces can be found in \cite[3.3]{engelking}.\medskip

Let $\overline{\B_{\mathbf{F}}}^{\Fo\left(X\right)}$ be the closure of $\B_{\mathbf{F}}$ in $\Fo\left(X\right)$. We will need the following characterization of (weakly) compactly embedded NSCF's (part (i) was essentially proven in \cite[Theorem 2.3]{erz2}; part (ii) is a variation of a classic result, see \cite{bartle}, \cite[VI.7, Theorem 1]{ds}, \cite[3.7, Theorem 5]{grot}, \cite{wada}).

\begin{theorem}\label{barrell} Let $\mathbf{F}$ be a NSCF over $X$. Then:
\item[(i)] $\mathbf{F}$ is weakly compactly embedded if and only if $\overline{\B_{\mathbf{F}}}^{\Fo\left(X\right)}\subset \Co\left(X\right)$.
\item[(ii)] If $X$ is compactly generated, then $\mathbf{F}$ is compactly embedded if and only if $\B_{\mathbf{F}}$ is equicontinuous and if and only if the correspondence $x\to x_{\mathbf{F}}$ is norm-continuous.
\end{theorem}

\begin{example}\label{dual}
Let $F$ be a separable non-reflexive normed space. If $X=\Bo_{F^{*}}$ is endowed with the weak* topology, then $X$ is a metrizable compact space (see \cite[V.5, Theorem 1]{ds}), and $F\subset\Co\left(X\right)$. Let $\mathbf{F}$ stand for $F$ considered as a BSCF over $X$. It follows from Alaoglu and Goldstine theorems (see \cite[V.4, theorems 2 and 5]{ds}) that $\overline{\B_{\mathbf{F}}}^{\Fo\left(X\right)}=\overline{\B_{F}}^{\sigma\left(\left(F^{*}\right)',F^{*}\right)}=\Bo_{F^{**}}\not\subset\Co\left(X\right)$. Hence, $\mathbf{F}$ is not weakly compactly embedded.
\qed\end{example}\smallskip

\begin{example}[Lipschitz space]\label{lip}
Let $\rho$ be a metric on $X$ and let $z\in X$. For $f:X\to\C$ define $\dil f= \sup\left\{\frac{\left|f\left(x\right)-f\left(y\right)\right|}{\rho\left(x,y\right)}\left|x,y\in X,~ x\ne y\right.\right\}$. This functional generates a BSCF $Lip\left(X,\rho\right)=\left\{f:X\to\C\left|\dil f<+\8\right.\right\}$ with the norm $\|f\|=\dil f+\left|f\left(z\right)\right|$. One can show that $\|x_{\mathbf{F}}\|=\max\left\{1,\rho\left(x,z\right)\right\}$ and $\|x_{\mathbf{F}}-y_{\mathbf{F}}\|=\rho\left(x,y\right)$, for every $x,y\in X$ (the proof is a slight modification of the proof from \cite{ae}). Hence, $Lip\left(X,\rho\right)$ is compactly embedded due to part (ii) of Theorem \ref{barrell}.
\qed\end{example}\smallskip

Now let us move towards the stated goal of the section. We will call a sequence $\left\{f_{n}\right\}_{n\in\N}\subset\Co\left(X\right)$ \emph{very independent}, if there is no non-zero sequence $\left\{a_{n}\right\}_{n\in\N}\subset\C$ such that $\sum\limits_{n\in\N}a_{n}f_{n}\equiv0$, where the series converges in $\Co\left(X\right)$. We will also call $\left\{f_{n}\right\}_{n\in\N}$ \emph{tempered} if there is an increasing sequence $\left\{U_{n}\right\}_{n\in\N}$ of open sets such that $X=\bigcup\limits_{n\in\N}U_{n}$ and $\|f_{n}\|_{\8}^{U_{n}}<+\8$.

\begin{example}
Assume that $X$ is $\sigma$-compact and locally compact. Then, there is an increasing sequence $\left\{U_{n}\right\}_{n\in\N}$ of open relatively compact sets, such that $X=\bigcup\limits_{n\in\N}U_{n}$ (see \cite[3.8.C]{engelking}). Every sequence $\left\{f_{n}\right\}_{n\in\N}\subset\Co\left(X\right)$ is therefore tempered, since a continuous function is always bounded on a relatively compact set.
\qed\end{example}\smallskip

\begin{example}
Let $\rho$ be a metric on $X$ and assume that $f_{n}$ is Lipschitz with respect to $\rho$, for every $n\in \N$. Then, $\left\{f_{n}\right\}_{n\in\N}$ is tempered. Indeed, fix any $z\in X$, and let $U_{n}=\B\left(z,n\right)$; then $\|f_{n}\|_{\8}^{U_{n}}\le \left|f_{n}\left(z\right)\right|+n \dil f_{n}<+\8$.
\qed\end{example}\smallskip

\begin{example}
Assume that $X$ is a domain in $\C^{n}$. Then, the collection of all monomials is very independent. It is also tempered, since $X$ is $\sigma$-compact and locally compact.
\qed\end{example}\smallskip

As in \cite{mr} we recall (see \cite[Proposition 1f3 and Theorem 1f4]{lt}) that if $E$ is a separable Banach space, there is a sequence $\left\{e_{n}\right\}_{n\in\N}\subset \partial \B_{E}$ such that $\spac \left\{e_{n}\right\}_{n\in\N}=E$, and a bounded  sequence $\left\{\nu_{n}\right\}_{n\in\N}\subset E^{*}$ such that $\left\{\nu_{n}\right\}_{n\in\N}^{\bot}=\left\{0_{E}\right\}$ and $\left<e_{n},\nu_{m}\right>=\delta_{mn}$, for $m,n\in\N$. Moreover, if $E^{*}$ is separable, the latter sequence can be chosen so that $\spac \left\{\nu_{n}\right\}_{n\in\N}=E^{*}$.

\begin{proposition}\label{tvi}
Let $F$ be a separable Banach space and let $X$ be compactly generated. Let $\left\{g_{n}\right\}_{n\in\N}\subset\Co\left(X\right)$ be very independent and tempered, with the corresponding $\left\{U_{n}\right\}_{n\in\N}$. Let $\left\{b_{n}\right\}_{n\in\N}\subset\left(0,+\8\right)$ be such that $\sum\limits_{n\in\N}\frac{\|g_{n}\|_{\8}^{U_{n}}}{b_{n}}<+\8$. Then there is a compactly embedded BSCF $\mathbf{F}$ over $X$, which is isometrically isomorphic to $F$ and such that $\spac\left\{g_{n}\right\}_{n\in\N}=\mathbf{F}\subset\spac^{\Co\left(X\right)}\left\{g_{n}\right\}_{n\in\N}$, and $\|g_{n}\|=b_{n}$, for every $n\in\N$.
\end{proposition}
\begin{proof}
The proof is similar to that in \cite{mr}. Fix $\left\{e_{n}\right\}_{n\in\N}\subset\partial \B_{F}$ and $\left\{\nu_{n}\right\}_{n\in\N}\subset \alpha \B_{F^{*}}$, as above, where $\alpha>0$. Let $J: F\to \Co\left(X\right)$ be defined by $Jf=\sum\limits_{n\in\N}  \frac{\left<f,\nu_{n}\right>}{b_{n}}g_{n}$, for $f\in F$. In order to prove that $J$ is well-defined we need to show that the series converges uniformly on compact sets. Let $K\subset X$ be compact. Since $\left\{U_{n}\right\}_{n\in\N}$ is an increasing sequence whose union covers $K$ there is $m\in\N$ such that $K\subset U_{m}$. Then for $f\in F$ we have
$$\sum\limits_{n\in\N} \left\|\frac{\left<f,\nu_{n}\right>}{b_{n}}g_{n}\right\|_{\8}^{K}\le \sum\limits_{n\in\N}\frac{\alpha\|f\|}{b_{n}}\|g_{n}\|_{\8}^{K}\le\alpha \|f\|\left(\sum\limits_{n=1}^{m-1}\frac{\|g_{n}\|_{\8}^{K}}{b_{n}}+\sum\limits_{n=m}^{\8}\frac{\|g_{n}\|_{\8}^{U_{n}}}{b_{n}}\right)<+\8.$$
Since $\Co\left(X\right)$ is complete, the series converges in $\Co\left(X\right)$. Moreover, $$JF=J\spac \left\{e_{n}\right\}_{n\in\N}\subset\spac^{\Co\left(X\right)}\left\{Je_{n}\right\}_{n\in\N}= \spac^{\Co\left(X\right)}\left\{g_{n}\right\}_{n\in\N}.$$ Assume that $Jf\equiv0$, for some $f$. Since $\left\{g_{n}\right\}_{n\in\N}$ is very independent, it follows that $\left<f,\nu_{n}\right>=0$, for every $n$, which implies that $f\in \left\{\nu_{n}\right\}_{n\in\N}^{\bot}=\left\{0_{F}\right\}$. Hence, $J$ is injective.

Now, identifying $F$ with $\mathbf{F}=JF$ we see that $x_{\mathbf{F}}=\sum\limits_{n\in\N}\frac{g_{n}\left(x\right)}{b_{n}} \nu_{n}$, for $x\in X$. The same estimates as above show that this series converges in $\Co\left(X,E^{*}\right)$, and so from part (ii) Theorem \ref{barrell}, $\mathbf{F}$ is compactly embedded. Finally, since $J$ is an isometry from $F$ onto $\mathbf{F}$ we have $\spac\left\{g_{n}\right\}_{n\in\N}=\mathbf{F}$ and $\|g_{n}\|=\|b_{n}e_{n}\|=b_{n}$, for every $n\in\N$.
\end{proof}

Note that $b_{n}$'s as in the proposition can always be chosen. For example, take $b_{n}=2^{n}\|g_{n}\|_{\8}^{U_{n}}$, for $n\in\N$. Similarly, one can show that every separable Banach space can be realized as a subspace of a given BSCF.

\begin{proposition}\label{tvi1}
Let $F$ be a separable Banach space and let $\mathbf{E}$ be a BSCF over $X$. Let $\left\{g_{n}\right\}_{n\in\N}\subset\mathbf{E}$ be a very independent sequence. Let $\left\{b_{n}\right\}_{n\in\N}\subset\left(0,+\8\right)$ be such that $\sum\limits_{n\in\N}\frac{\|g_{n}\|_{\mathbf{E}}}{b_{n}}<+\8$. Then there is a BSCF $\mathbf{F}$ over $X$, which is isometrically isomorphic to $F$ and such that $\left\{g_{n}\right\}_{n\in\N}\subset \mathbf{F}\subset\mathbf{E}$, with $\|g_{n}\|_{\mathbf{F}}=b_{n}$, for every $n\in\N$.
\end{proposition}

Let us now prove the main result of the section.

\begin{theorem}\label{real}
Let $F$ be a separable Banach space and let $\left(X,\rho\right)$ be a separable metric space. Then there is BSCF $\mathbf{F}$ over $X$ that has the following properties:
\item[(i)] $\mathbf{F}$ is isometrically isomorphic to $F$, compactly embedded and consists of Lipschitz functions.
\item[(ii)] The point evaluations form a linearly independent subset of $\mathbf{F}^{*}$.
\item[(iii)] $\mathbf{F}$ generates the topology of $X$, i.e. the topology of $X$ is the minimal topology which makes every element of $\mathbf{F}$ continuous.\smallskip

Moreover, if $F^{*}$ is separable, $\mathbf{F}$ can be chosen so that the weak, pointwise and compact-open topologies coincide on $\Bo_{\mathbf{F}}$.
\end{theorem}
\begin{proof}
Let $Y=\left\{y_{n}\right\}_{n\in\N}\subset X$ be a dense set of distinct points. Let $g_{0}=\1$, and for $n\in\N$ set $g_{n}=\rho\left(\cdot,\left\{y_{1},...,y_{n}\right\}\right)\wedge 1$. Clearly, $g_{n}$ is a Lipschitz function with $\|g_{n}\|_{\8}\le 1$ and $g_{n}^{-1}\left(0\right)=\left\{y_{k}\right\}_{k=1}^{n}$, for every $n\in\N_{0}=\N\cup \left\{0\right\}$.

Let us prove that $\left\{g_{n}\right\}_{n\in\N_{0}}$ is a very independent sequence. Assume that $\left\{a_{n}\right\}_{n\in\N_{0}}\subset\C$ are such that $\sum\limits_{n\in\N_{0}}a_{n}g_{n}\equiv0$, and the series converges. Let us show by induction that $a_{k}=0$, for all $k\in\N_{0}$. Indeed, if the statement is proven for $k=0,...,n-1$ (if $n=0$ nothing is proven), we have
$0=\sum\limits_{m\in\N_{0}}a_{m}g_{m}\left(y_{n+1}\right)=\sum\limits_{m=n}^{+\8}a_{m}g_{m}\left(y_{n+1}\right)=a_{n}g_{n}\left(y_{n+1}\right)$, as $g_{m}\left(y_{n+1}\right)=0$, for $m>n$. Since $g_{n}\left(y_{n+1}\right)\ne 0$ we conclude that $a_{n}=0$.\medskip

Let $\mathbf{F}$ be a BSCF that is generated by the preceding proposition and $\mathbf{E}=Lip\left(X,\rho\right)$. Then $\mathbf{F}$ is isometrically isomorphic to $F$, contains $g_{n}$, for every $n\in\N_{0}$ and is included into the Lipschitz space. Since the latter is compactly embedded, the same is true for $\mathbf{F}$.\medskip

In order to prove (ii), we will show that if $x_{1},...,x_{n}\in X$ are distinct, and $a_{1},...,a_{n}\in\C$ are such that $\sum\limits_{i=1}^{n}a_{i}g_{m}\left(x_{i}\right)=0$, for every $m\in \N_{0}$, then $a_{1}=...=a_{n}=0$. First, consider the case when $x_{1},...,x_{n}\in Y$, i.e. $x_{i}=y_{k_{i}}$, for every $i\in\overline{1,n}$. We may assume that $k_{1}<...<k_{n}$. Assume that we have proved that $a_{l+1}=...=a_{n}=0$, for $l\in\overline{1,n}$ (if $l=n$, then nothing is proven). Then, for $m=k_{l}-1$ we have that $0=\sum\limits_{i=1}^{n}a_{i}g_{m}\left(x_{i}\right)=\sum\limits_{i=1}^{l}a_{i}g_{m}\left(y_{k_{i}}\right)=a_{l}g_{m}\left(y_{m+1}\right)$, since $k_{i}\le m$, for $i\in \overline{1,l-1}$. As $g_{m}\left(y_{m+1}\right)\ne 0$ we conclude that $a_{l}=0$.

Let us now prove the claim in the general case by induction. Since $g_{0}=\1$, the claim is true for $n=1$. Assume that it is proven for $n-1$. The case when $x_{1},...,x_{n}\in Y$ was already considered, and so we may assume that $x_{n}\not\in Y$. Since $Y$ is dense, $\left\{g_{m}\left(x_{i}\right)\right\}_{m\in\N_{0}}$ is decreasing to $0$, for every $i\in\overline{1,n}$. Let $m$ be such that for every $i\in\overline{1,n}$ we have $g_{m}\left(x_{i}\right)\le \frac{d}{2}$, where $d=\min\limits_{0\le i<j\le n}\rho\left(x_{i},x_{j}\right)$.
 
Since $x_{n}\not\in Y$, it follows that $0\not\in \left\{g_{m}\left(x_{n}\right)\right\}_{m\in\N_{0}}$, but since $g_{m}\left(x_{n}\right)\xrightarrow{n\to\8}0$, there is $l> m$ such that  $g_{l}\left(x_{n}\right)< g_{l-1}\left(x_{n}\right)$. Since $g_{l}\left(x_{n}\right)= g_{l-1}\left(x_{n}\right)\wedge \rho\left(x_{n},y_{l}\right)$, it follows that $\rho\left(x_{n},y_{l}\right)<g_{l-1}\left(x_{n}\right)\le \frac{d}{2}$.

For $i\in\overline{1,n-1}$ we have $\rho\left(x_{i},y_{l}\right)\ge \rho\left(x_{i},x_{n}\right)-\rho\left(x_{n},y_{l}\right)\ge d-\frac{d}{2}=\frac{d}{2}\ge g_{l-1}\left(x_{i}\right)$. Hence, $g_{l}\left(x_{i}\right)=g_{l-1}\left(x_{i}\right)\wedge \rho\left(x_{i},y_{l}\right)=g_{l-1}\left(x_{i}\right)$. We have $-a_{n}g_{l-1}\left(x_{n}\right)=\sum\limits_{i=1}^{n-1}a_{i}g_{l-1}\left(x_{i}\right)=\sum\limits_{i=1}^{n-1}a_{i}g_{l}\left(x_{i}\right)=-a_{n}g_{l}\left(x_{n}\right)$. Since $g_{l}\left(x_{n}\right)\ne g_{l-1}\left(x_{n}\right)$ we conclude that $a_{n}=0$, but then $\sum\limits_{i=1}^{n-1}a_{i}g_{m}\left(x_{i}\right)=0$, for every $m\in N_{0}$, which contradicts the assumption of induction.\medskip

Let us prove (iii). Let $\tau$ be the topology generated by $\mathbf{F}$. Let $x\in X$ and let $\varepsilon \in\left(0,1\right)$. Let $n\in \N$ be minimal such that $\rho\left(x,y_{n}\right)<\frac{\varepsilon}{2}$ (its existence follows from the density of $Y$). Then, $g_{n-1}\left(x\right)\ge \frac{\varepsilon}{2}>g_{n}\left(x\right)$. Since $g_{n-1}$ and $g_{n}$ are continuous with respect to $\tau$, this topology contains the set $U=\left\{z\in X\left|g_{n-1}\left(z\right)>g_{n}\left(z\right)<\frac{\varepsilon}{2}\right.\right\}$. Clearly, $x\in U$, and for every $z\in U$ we have $g_{n-1}\left(z\right)>g_{n}\left(z\right)=g_{n-1}\left(z\right)\wedge \rho\left(z,y_{n}\right)$, from where $\rho\left(z,y_{n}\right)=g_{n}\left(z\right)<\frac{\varepsilon}{2}$. Hence, $\rho\left(x,z\right)\le \rho\left(x,y_{n}\right)+\rho\left(z,y_{n}\right)<\varepsilon$ and so $x\in U\subset \B\left(x,\varepsilon\right)$. Thus, since $x$ and $\varepsilon$ were chosen arbitrarily, the elements of $\tau$ form a basis of the original topology on $X$, and so $\mathbf{F}$ generates the topology of $X$.\medskip

Let us prove the last claim. As was mentioned above, if $\mathbf{F}^{*}$ is separable, we may assume that for $x\in X$ we have $x_{\mathbf{F}}=\sum\limits_{n\in\N_{0}}\frac{g_{n}\left(x\right)}{b_{n}} \nu_{n}$, where $\left\{b_{n}\right\}_{n\in\N}\subset\left(0,+\8\right)$ and $\left\{\nu_{n}\right\}_{n\in\N_{0}}\subset E^{*}$ are such that $\spac \left\{\nu_{n}\right\}_{n\in\N_{0}}=E^{*}$. Using induction, and consecutive substitution of $x=y_{n}$, one can show that $\nu_{n}\in \spa \left\{y_{\mathbf{F}}\left|y\in Y\right.\right\}$, for every $n\in\N_{0}$. Hence, $\spac \left\{x_{\mathbf{F}}\left|x\in X\right.\right\}=\mathbf{F}^{*}$, and so, as was mentioned in Remark \ref{sde}, the weak and pointwise topologies coincide on $\Bo_{\mathbf{F}}$. The compact-open topology is stronger than the pointwise topology, but since $\mathbf{F}$ is compactly embedded, it follows that the compact-open topology is weaker than the weak topology on $\Bo_{\mathbf{F}}$ (see \cite[2.18, Theorem 12]{grot}). Thus, all these three topologies coincide on $\Bo_{\mathbf{F}}$.
\end{proof}

\begin{remark}
The theorem can be generalized to the level of Frechet spaces, using the result from \cite{bo}.
\qed\end{remark}

\section{Multipliers of a NSCF}

Everywhere in this section $X$ is a Hausdorff space. It turns out that an analogue of Proposition \ref{muw} holds in the context of NSCF's.

\begin{theorem}\label{mus}Let $\mathbf{F}$ and $\mathbf{E}$ be NSCF's over $X$. If $\mathbf{F}$ is $1$-independent, then:
\item[(i)] $Mult\left(\mathbf{F}, \mathbf{E}\right)$ is a NSCF over $X$, and moreover the strong operator topology is stronger than the compact-open topology on $Mult\left(\mathbf{F}, \mathbf{E}\right)$.
\item[(ii)] If $\mathbf{E}$ is weakly compactly embedded, then so is $Mult\left(\mathbf{F}, \mathbf{E}\right)$.
\item[(iii)] If $X$ is compactly generated, and $\mathbf{E}$ is compactly embedded, then so is $Mult\left(\mathbf{F}, \mathbf{E}\right)$.
\end{theorem}
\begin{proof}
(i): The proof of the fact that $Mult\left(\mathbf{F},\mathbf{E}\right)\subset \Co\left(X\right)$ is analogous to \cite[Proposition 2.2]{erz3}. Let $K\subset X$ be compact. For every $x\in K$ let $f_{x}\in \mathbf{F}$ be such that $f_{x}\left(x\right)=2$, and let $U_{x}$ be a neighborhood of $x$ such that $\left|f_{x}\left(y\right)\right|>1$, for every $y\in U_{x}$. Since $K$ is compact, we can choose $x_{1},...,x_{n}$ such that $\bigcup\limits_{i=1}^{n}U_{x_{i}}=X$. Denote $f_{i}=f_{x_{i}}$ and $U_{i}=U_{x_{i}}$, for $i\in\overline{1,n}$. The seminorm $|||\omega|||=\sum\limits_{i=1}^{n}\|\omega f_{i}\|_{\mathbf{E}}$, $\omega\in Mult\left(\mathbf{F}, \mathbf{E}\right)$ is continuous with respect to the strong operator topology. Since $\mathbf{E}$ is a NSCF over $X$, there is $a>0$ such that $a\|\cdot\|_{\8}^{K}\le \|\cdot\|_{\mathbf{E}}$. Hence,
$$|||\omega|||\ge a\sum\limits_{i=1}^{n}\|\omega f_{i}\|_{\8}^{K}\ge a\sum\limits_{i=1}^{n}\|\omega f_{i}\|_{\8}^{K\cap U_{i}}\ge a\sum\limits_{i=1}^{n}\|\omega\|^{K\cap U_{i}}\ge a\bigvee_{i=1}^{n}\|\omega\|^{K\cap U_{i}}=a\|\omega\|^{K}_{\8},$$
and so $\|\cdot\|^{K}_{\8}$ is continuous on $Mult\left(\mathbf{F}, \mathbf{E}\right)$ with respect to the strong operator topology. Since $K$ was chosen arbitrarily, the strong operator topology on $Mult\left(\mathbf{F}, \mathbf{E}\right)$ is stronger than the compact-open topology. Since the norm topology is stronger than the strong operator topology, we conclude that $Mult\left(\mathbf{F}, \mathbf{E}\right)$ is a NSCF.\medskip

(ii): By definition, $\left\{\omega f\left|\omega\in \B_{Mult\left(\mathbf{F},\mathbf{E}\right)},~ f\in \B_{\mathbf{F}}\right.\right\}\subset \B_{\mathbf{E}}$. Since multiplication of functions is continuous on $\Fo\left(X\right)$ it follows that $$\left\{\omega f\left|\omega\in \overline{\B_{Mult\left(\mathbf{F},\mathbf{E}\right)}}^{\Fo\left(X\right)},~ f\in \B_{\mathbf{F}}\right.\right\}\subset \overline{\B_{\mathbf{E}}}^{\Fo\left(X\right)}\subset\Co\left(X\right),$$ where the latter inclusion follows from part (i) of Theorem \ref{barrell}. Let $\omega\in \overline{\B_{Mult\left(\mathbf{F},\mathbf{E}\right)}}^{\Fo\left(X\right)}$, let $x\in X$ and let $f\in \B_{\mathbf{F}}$ be such that $f\left(x\right)\ne 0$. Then, since both $\omega f$ and $ f$ are continuous, it follows that their quotient $\omega$ is continuous at $x$, and as $x$ was chosen arbitrarily we get that $\omega$ is continuous. Since $\omega$ was chosen arbitrarily, we conclude that $\overline{\B_{Mult\left(\mathbf{F},\mathbf{E}\right)}}^{\Fo\left(X\right)}\subset\Co\left(X\right)$, and so from part (i) of Theorem \ref{barrell}, $Mult\left(\mathbf{F},\mathbf{E}\right)$ is weakly compactly embedded.\medskip

(iii): From part (ii) of Theorem \ref{barrell} $\B_{\mathbf{E}}$ is equicontinuous, and since  $Mult\left(\mathbf{F},\mathbf{E}\right)$ is a NSCF, according to (i), in order to prove that it is compactly embedded, it is enough to show that $\B_{Mult\left(\mathbf{F},\mathbf{E}\right)}$ is equicontinuous. Let $\varepsilon>0$ and let $x\in X$. Denote $b=\|x_{\mathbf{E}}\|=\sup\limits_{g\in\B_{\mathbf{E}}}\left|g\left(x\right)\right|$ and let $f\in \B_{\mathbf{F}}$ be such that $f\left(x\right)=a>0$.

For fixed $\delta,\eta>0$ let $U$ be an open neighborhood of $x$ such that $\left|f\left(y\right)-f\left(x\right)\right|<\delta $ and $\left|g\left(y\right)-g\left(x\right)\right|<\eta $, for every $y\in U$ and $g\in \B_{\mathbf{E}}$. Then for every $\omega\in \B_{Mult\left(\mathbf{F},\mathbf{E}\right)}$ we have $g=\omega f\in \B_{\mathbf{E}}$, and so for $y\in U$ we get
\begin{align*}
\left(a-\delta \right)\left|\omega\left(y\right)\right|&\le\left(\left|f\left(x\right)\right|-\left|f\left(x\right)-f\left(y\right)\right|\right)\left|\omega\left(y\right)\right| \\ &\le\left|\omega\left(y\right)f\left(y\right)\right|=\left|g\left(y\right)\right|\le \left|g\left(x\right)\right|+\left|g\left(y\right)-g\left(x\right)\right|\le b+ \eta,\end{align*}
which yields $\left|\omega\left(y\right)\right|\le \frac{b+\eta }{a-\delta }$ on $U$. Therefore,
\begin{align*}
\eta &>\left|g\left(y\right)-g\left(x\right)\right|=\left|f\left(y\right)\omega\left(y\right)-f\left(x\right)\omega\left(x\right)\right|\\
&=\left|\left(f\left(y\right)-f\left(x\right)\right)\omega\left(y\right)+\left(\omega\left(y\right)-\omega\left(x\right)\right)f\left(x\right)\right|\\
&\ge \left|\omega\left(y\right)-\omega\left(x\right)\right|\left|f\left(x\right)\right|-\left|f\left(y\right)-f\left(x\right)\right|\left|\omega\left(y\right)\right|\ge a\left|\omega\left(y\right)-\omega\left(x\right)\right|-\delta\frac{b+\eta }{a-\delta },
\end{align*}
from where, setting $\eta= \frac{a\varepsilon}{3}$ and $\delta=\frac{a^{2}\varepsilon}{3b}\wedge \frac{a}{3}$ (if $b=0$ set $\delta= \frac{a}{3}$) we get
$$\left|\omega\left(y\right)-\omega\left(x\right)\right|< \frac{\eta}{a}+
\frac{\delta b+\delta\eta }{a\left(a-\delta \right)}=\frac{\delta b+ a\eta }{a\left(a-\delta \right)}\le \frac{a^{2}\varepsilon+ a^{2}\varepsilon }{2 a^{2}}= \varepsilon.$$

Thus, since $\varepsilon$, $x$ and $\omega$ were chosen arbitrarily, $\B_{Mult\left(\mathbf{F}\right)}$ is equicontinuous.
\end{proof}

Let us now discuss multiplier algebras of NSCF's. If $\mathbf{F}$ is a $1$-independent NSCF over $X$, then $Mult\left(\mathbf{F}\right)$ contractively embeds into $\Co_{\8}\left(X\right)$ (see \cite[Proposition 2.2]{erz3}), and it follows from Theorem \ref{mus} that if $\mathbf{F}$ is weakly compactly embedded, then so is $Mult\left(\mathbf{F}\right)$. If $X$ is compactly generated and $\mathbf{F}$ is compactly embedded, then so is $Mult\left(\mathbf{F}\right)$; the converses to the last two statements are false (see Example \ref{ahn}).

If $\dim Mult\left(\mathbf{F}\right)<\8$, then every multiplier of $\mathbf{F}$ is constant on every component of $X$. Indeed, if $Y$ is a component of $X$ and a (continuous) multiplier $\omega$ is not a constant on $Y$,  then $\omega\left(Y\right)$ is of infinite cardinality, and so $Mult\left(\mathbf{F}\right)$ separates infinite number of points. Hence, the point evaluations at these points are linearly independent due to Dedekind's theorem, and so $\dim Mult\left(\mathbf{F}\right)'=\8$.\medskip

The last observation motivates the following definition. We will call a NSCF \emph{anti-multiplicative} if it admits only constant multipliers. It follows that if $X$ is connected, a NSCF $\mathbf{F}$ over $X$ is anti-multiplicative as soon as $\dim Mult\left(\mathbf{F}\right)<\8$. Let us start with an example.

\begin{example}\label{ardy}
Consider the Hardy space $\Ho^{2}$ over $\D$ (see Example \ref{hardy}), and define functions $p_{n}:\D\to\C$, $n\in\N_{0}$ by $p_{0}\left(z\right)=e^{z}$ and $p_{n}\left(z\right)=z^{n}$, $n\in\N$. Let $\mathbf{F}=\spa\left\{ p_{n}, n\in\N_{0}\right\}$, which is a dense subspace of $\Ho^{2}$. Assume $\omega\in Mult\left(\mathbf{F}\right)\backslash\left\{0\right\}$. There are $a_{0},a_{1},...,a_{n}\in\C$ such that $\omega\cdot p_{1}=a_{0}p_{0}+a_{1}p_{1}+...+a_{n}p_{n}$. Then, since $0=\left[\omega\cdot p_{1}\right]\left(0\right)=a_{0}$ we conclude that $\omega=a_{1}+a_{2}p_{1}+...+a_{n}p_{n-1}$ is a polynomial. There are $b_{0},b_{1},...,b_{m}\in\C$ such that $\omega\cdot p_{0}=b_{0}p_{0}+b_{1}p_{1}+...+b_{m}p_{m}$, from where $\left(\omega-b_{0}\right)\cdot p_{0}$ is a polynomial. Since $p_{0}$ is not a rational function, it follows that $\left(\omega-b_{0}\right)\cdot p_{0}\equiv 0$, from where $\omega\equiv b_{0}$ is a constant function.
\qed\end{example}

To construct anti-multiplicative BSCF's we need an auxiliary result.

\begin{lemma}
Assume that $X$ is connected and let $\mathbf{F}$ be a $1$-independent BSCF over $X$ with $\dim \mathbf{F}=\8$. If $\mathbf{F}$ is such that every bounded operators on it is a sum of a scalar and a compact operator, then $\mathbf{F}$ is anti-multiplicative.
\end{lemma}
\begin{proof}
Let $\omega\in Mult\left(\mathbf{F}\right)$. Then, there is $\lambda\in\C$ such that $M_{\omega}-\lambda Id_{\mathbf{F}}$ is a compact operator on $\mathbf{F}$. But the latter operator is equal to $M_{\omega- \lambda\1}$. Since there can be no non-zero compact multiplication operator on an infinite-dimensional $1$-independent NSCF (see \cite[Proposition 2.10]{erz3}), it follows that $\omega\equiv\lambda$.
\end{proof}

\begin{proposition}\label{anm}
There is an anti-multiplicative compactly embedded $1$-independent infinite-dimensional BSCF over every connected separable metric space.
\end{proposition}
\begin{proof}
Let $F$ be the Argyros-Haydon space (see \cite{ah}), which is an infinite-dimensional separable non-reflexive Banach space such that every bounded operator on it is a sum of a scalar and a compact operators. From Theorem \ref{real}, there is a compactly embedded $1$-independent BSCF over $X$, which is isometrically isomorphic to $F$. From the preceding lemma it is anti-multiplicative.
\end{proof}

Using a similar idea it is possible to construct a non-weakly-compactly embedded NSCF whose multiplier algebra is compactly embedded.

\begin{example}\label{ahn} Let $F$ be the Argyros-Haydon space and construct $X$ and $\mathbf{F}$ as in Example \ref{dual}. We get a non-weakly compactly embedded NSCF over a metrizable connected compact space, but $Mult\left(\mathbf{F}\right)$ is one-dimensional, and therefore compactly embedded.
\end{example}

Consider a curios example of a NSHF whose multiplier algebra isomorphic to a Hilbert space.

\begin{example}
Let $\mathbf{H}$ be a HSHF over $\D$ consisting of all $f\in\Ho\left(\D\right)$ such that $f'\in \Ho^{2}$, with the norm $\|f\|^{2}=\|f'\|^{2}_{\Ho^{2}}+\left|f\left(0\right)\right|^{2}$. Clearly, this is a Hilbert space, and point evaluations are bounded since $\left|f\left(z\right)\right|\le \left|f\left(z\right)-f\left(0\right)\right| +\left|f\left(0\right)\right|  \le \left|z\right|\|f'\|_{\8}^{\left[0,z\right]}+\left|f\left(0\right)\right|$, and the latter semi-norms are majorated by $\|f\|$. Obviously, $\1\in\mathbf{H}$.

Let us show that $\mathbf{H}$ is an algebra. Since $\Ho^{2}\subset \Ho^{1}$ it follows that $\mathbf{H}\subset A\left(\D\right)\subset \Ho^{\8}$ (see \cite[Theorem 3.11]{duren}). For $f,g\in \mathbf{H}$ we have $\left(fg\right)'=fg'+f'g$. Since $f,g\in \Ho^{\8}=Mult\left(\Ho^{2}\right)$ it follows that $fg', f'g\in \Ho^{2}$, and so $fg\in \mathbf{H}$. Thus, from Proposition \ref{subn} $\mathbf{H}=Mult\left(\mathbf{H}\right)$ (as topological vector spaces).
\qed\end{example}\bigskip

We conclude the section with some remarks regarding the multipliers of NSHF's. First, similarly to the continuous case, if $X$ is a domain in $\C^{n}$, $\mathbf{F}$ is a $1$-independent NSHF over $X$, and $\mathbf{E}$ is a NSHF over $X$, then $Mult\left(\mathbf{F},\mathbf{E}\right)\subset \Ho\left(X\right)$. For the multiplier algebras an even stronger fact is true (see \cite[Proposition 4.3]{erz1}).

\begin{proposition}\label{hh}Assume that $X$ is a domain in $\C^{n}$, and let $\mathbf{F}$ be a NSHF over $X$. Then, $Mult\left(\mathbf{F}\right)$ contractively embeds into $\Ho_{\8}\left(X\right)$ in the sense that if $T\in\Lo\left(\mathbf{F}\right)$ is a MO, there is a unique holomorphic $\omega:X\to\C$ such that $T=M_{\omega}$, for which also $\|\omega\|_{\8}\le\|M_{\omega}\|$.
\end{proposition}

Since $\Ho_{\8}\left(\C^{n}\right)=\C$, applying Proposition \ref{tvi} to the monomials we get the following result.

\begin{corollary}\label{anm2}
Every NSHF over $\C^{n}$ is anti-multiplicative. For every separable Banach space $F$ there is an anti-multiplicative $1$-independent BSCF isometrically isomorphic to $F$.
\end{corollary}

Note that our examples of anti-multiplicative NSCF's exploit some kind of pathologies either of the phase space, or the normed space. It is natural to ask if one can find an example where both are as well-behaved as possible.

\begin{question}
Does there exist an anti-multiplicative HSHF over $\D$?
\end{question}

\section{Subalgebras of the multiplier algebras}

For the purposes of this section we need to introduce an additional property of NSF's. Let $\mathbf{F}$ be a NSF over a set $X$. Recall that $\overline{\B_{\mathbf{F}}}^{\Fo\left(X\right)}$ is the closure of $\B_{\mathbf{F}}$ in $\Fo\left(X\right)$, which is bounded, closed, convex and balanced.  Hence, $\widehat{\mathbf{F}}=\left\{\alpha f\left|\alpha>0,~ f\in \overline{\B_{\mathbf{F}}}^{\Fo\left(X\right)}\right.\right\}$ is a NSF over $X$ with the closed unit ball $\overline{\B_{\mathbf{F}}}^{\Fo\left(X\right)}$. One can show that $\widehat{\mathbf{F}}=\left(\spa \left\{x_{\mathbf{F}}\left|x\in X\right.\right\}\right)^{*}$ (as normed spaces) via the bilinear form induced by $\left<x_{\mathbf{F}},f\right>=f\left(x\right)$, and moreover, $\overline{\B_{\mathbf{F}}}^{\Fo\left(X\right)}= J_{\mathbf{F}}^{**}\Bo_{\mathbf{F}^{**}}$ (see \cite[Theorem 2.3]{erz2} and its proof in the case when $X$ is a discrete topological space). We will say that $\mathbf{F}$ is \emph{regular} if $\widehat{\mathbf{F}}=\mathbf{F}$ (as normed spaces), i.e. $\Bo_{\mathbf{F}}$ is closed in $\Fo\left(X\right)$.

While regularity of a NSF can be viewed as a type of maximality, the opposite property is to have point evaluations dense in the dual (as was mentioned in Remark \ref{sde}, $J_{\mathbf{F}}^{**}$ is injective if and only if $\spac \left\{x_{\mathbf{F}}\left|x\in X\right.\right\}=\mathbf{F}^{*}$). Unsurprisingly, these two properties combined are equivalent to reflexivity.

\begin{proposition}A NSF $\mathbf{F}$ over a set $X$ is reflexive if and only if it is regular and $\spac \left\{x_{\mathbf{F}}\left|x\in X\right.\right\}=\mathbf{F}^{*}$.
\end{proposition}

It follows from part (i) of Theorem \ref{barrell} that a NSCF $\mathbf{F}$ over a topological space $X$ is regular if and only if it is weakly compactly embedded and $\Bo_{\mathbf{F}}$ is closed in $\Co\left(X\right)$. It also follows from the preceding proposition that every reflexive NSCF over a separable topological space is separable.

\begin{example}
Since $\mathbf{F}$ from Example \ref{dual} is not weakly compactly embedded, it is not regular. On the other hand, $\overline{\B_{\mathbf{F}}}^{\Fo\left(X\right)}\cap \Co\left(X\right)=\Bo_{\mathbf{F}}$ is closed in $\Co\left(X\right)$.
\qed\end{example}\smallskip

\begin{example}\label{wuh}
If $X$ is a domain in $\C^{n}$, let $\Ho_{u}^{0}$ be the subspace of $\Ho_{u}^{\8}$ that consists of functions $f$ such that $\left|uf\right|$ vanishes at infinity. Under some mild conditions $\left(\Ho^{0}_{v}\right)^{**}=\Ho^{\8}_{v}$ with $J_{\Ho^{\8}_{v}}=J_{\Ho^{0}_{v}}^{**}$ (see \cite{bor} and the reference therein), and since in this case $\Ho^{0}_{v}\ne \Ho^{\8}_{v}$, it follows that the former is not regular, while the latter is.
\qed\end{example}\smallskip

For $X\subset\C^{n}$ let $\Po\left(X\right)$ be the linear space of polynomials of $n$ variable, viewed as functions on $X$.

\begin{example}\label{ad}
Let $X$ be a bounded domain in $\C^{n}$. Under some assumptions about $X$ we have $\overline{\B_{A\left(X\right)}}^{\Fo\left(X\right)}=\Bo_{\Ho_{\8}\left(X\right)}$ (see e.g. \cite[Theorem 6.4, and Remark 6.5]{range} and \cite{dgg}), and in particular $A\left(X\right)$ is not regular. The simplest example is when $X$ is convex, so that every $f\in \Bo_{\Ho_{\8}\left(X\right)}$ is approximated by $f_{n}\in   \Bo_{A\left(X\right)}$ defined by $f_{n}\left(X\right)=f\left(\frac{nx}{n+1}\right)$. Since polynomials are dense in $A\left(X\right)$ it follows that $\overline{\B_{\Po\left(X\right)}}^{\Fo\left(X\right)}=\Bo_{\Ho_{\8}\left(X\right)}$.\qed\end{example}\smallskip

It turns out that regularity is inherited by the multiplier spaces.

\begin{proposition}\label{muw2}Let $\mathbf{F}$ and $\mathbf{E}$ be NSF's over a set $X$. If $\mathbf{F}$ is $1$-independent, then:
\item[(i)]  If $\mathbf{E}$ is regular, then so is $Mult\left(\mathbf{F}, \mathbf{E}\right)$.
\item[(ii)] $Mult\left(\mathbf{F}, \mathbf{E}\right)\subset Mult\left(\widehat{\mathbf{F}}, \widehat{\mathbf{E}}\right)$, and the inclusion is contractive.
\end{proposition}
\begin{proof}
(i): Since multiplication is a continuous operation on $\Fo\left(X\right)$, it follows that if $\Bo_{\mathbf{E}}$ is closed in $\Fo\left(X\right)$, then $\Bo_{Mult\left(\mathbf{F},\mathbf{E}\right)}=\bigcap\limits_{f\in \Bo_{\mathbf{F}}} \left\{\omega\in \Fo\left(X\right)\left|\omega\cdot f\in \Bo_{\mathbf{E}}\right.\right\}$ is also closed.\medskip

(ii): If $\|\omega\|_{Mult\left(\mathbf{F},\mathbf{E}\right)}\le 1$, then $M_{\omega}\Bo_{\mathbf{F}}\subset \Bo_{\mathbf{E}}$. Since $M_{\omega}$ is a continuous operator on $\Fo\left(X\right)$, we have $M_{\omega}\overline{\B_{\mathbf{F}}}^{\Fo\left(X\right)}\subset\overline{\B_{\mathbf{E}}}^{\Fo\left(X\right)}$, from where $\|\omega\|_{Mult\left(\widehat{\mathbf{F}}, \widehat{\mathbf{E}}\right)}\le 1$. Hence, $Mult\left(\mathbf{F}, \mathbf{E}\right)$ contractively embeds into $Mult\left(\widehat{\mathbf{F}}, \widehat{\mathbf{E}}\right)$.
\end{proof}

Let $\mathbf{F}$ be a $1$-independent NSF over a set $X$. If $\mathbf{F}$ is regular, then so is $Mult\left(\mathbf{F}\right)$. The converse is false as demonstrated by the little Weighted space of holomorphic functions, which is a non-regular BSHF whose multiplier algebra is a regular NSHF $\Ho_{\8}\left(X\right)$.

It follows that $Mult\left(\mathbf{F}\right)\subset Mult\left(\widehat{\mathbf{F}}\right)$, contractively, and the latter is regular. Hence, $\1\in Mult \left(Mult\left(\mathbf{F}\right), Mult\left(\widehat{\mathbf{F}}\right)\right)\subset Mult \left(\widehat{Mult\left(\mathbf{F}\right)}, \widehat{Mult\left(\widehat{\mathbf{F}}\right)}\right)$, and so $\widehat{Mult\left(\mathbf{F}\right)}\subset \widehat{Mult\left(\widehat{\mathbf{F}}\right)}=Mult\left(\widehat{\mathbf{F}}\right)$ contractively. This observation motivates the following question.

\begin{question}
Is it true that $Mult\left(\widehat{\mathbf{F}}\right)=\widehat{Mult\left(\mathbf{F}\right)}$, for a $1$-independent BSF $\mathbf{F}$?
\end{question}

Note that without completeness the answer is negative as demonstrated by $\mathbf{F}$ from Example \ref{ardy}:  $Mult\left(\mathbf{F}\right)$ consists of constant functions, while the multiplier algebra of $\widehat{\mathbf{F}}=\Ho^{2}$ is $\Ho^{\8}$.\bigskip

Let $X$ be a set, let $\omega_{1},...,\omega_{n}:X\to\C$ and let $\vec{\omega}=\left(\omega_{1},...,\omega_{n}\right):X\to\C^{n}$. If $Y\subset \C^{n}$ is such that $\vec{\omega}\left(X\right)\subset Y$, define the composition operator $C_{\vec{\omega}}:\Fo\left(Y\right)\to\Fo\left(X\right)$ by $C_{\vec{\omega}}g=g\circ\vec{\omega}$. Clearly, $C_{\vec{\omega}}$ is continuous with respect to the pointwise topologies on these spaces. If $\omega_{1},...,\omega_{n}$ all belong to a certain algebra $\mathbf{E}$ of functions, then $C_{\vec{\omega}}$ is a homomorphism from $\Po\left(Y\right)$ into $\mathbf{E}$. Let us study this map in the case of a multiplier algebra.

\begin{proposition}\label{vna}
Let $\mathbf{F}$ be a regular $1$-independent NSF over a set $X$ and let $\omega_{1},...,\omega_{n}\in Mult\left(\mathbf{F}\right)$. Assume that $Y\subset \C^{n}$ is a domain as described in Example \ref{ad} and such that $\vec{\omega}\left(X\right)\subset \overline{Y}$. Then:
\item[(i)] If $C_{\vec{\omega}}:\Po\left(\overline{Y}\right)\to Mult\left(\mathbf{F}\right)$ is bounded with respect to $\|\cdot\|_{\8}^{\overline{Y}}$, then $C_{\vec{\omega}}$ is a bounded operator from $A\left(Y\right)$ into $Mult\left(\mathbf{F}\right)$.
\item[(ii)] If additionally $\vec{\omega}\left(X\right)\subset Y$, then $C_{\vec{\omega}}$ is a bounded operator from $\Ho_{\8}\left(Y\right)$ into $Mult\left(\mathbf{F}\right)$.
\end{proposition}
\begin{proof}
Let us start with (ii). As was mentioned in Example \ref{ad} we have $\overline{\B_{\Po\left(Y\right)}}^{\Fo\left(Y\right)}=\Bo_{\Ho_{\8}\left(Y\right)}$. From our assumption $C_{\vec{\omega}}\B_{\Po\left(Y\right)}\subset \alpha \Bo_{Mult\left(\mathbf{F}\right)}$, for some $\alpha$. Since $\mathbf{F}$ is regular, from part (i) of Proposition \ref{muw2} the same is true about $Mult\left(\mathbf{F}\right)$, and so $\alpha \Bo_{Mult\left(\mathbf{F}\right)}$ is pointwise compact. Recall that $C_{\vec{\omega}}$ is continuous with respect to the pointwise topology, and so $C_{\vec{\omega}}\Bo_{\Ho_{\8}\left(Y\right)}\subset \overline{C_{\vec{\omega}}\B_{\Po\left(Y\right)}} ^{\Fo\left(X\right)} \subset \alpha \Bo_{Mult\left(\mathbf{F}\right)}$. Hence, $C_{\vec{\omega}}\Ho_{\8}\left(Y\right)\subset Mult\left(\mathbf{F}\right)$, and so from the Closed Graph theorem, the claim follows.

The proof of (i) is done similarly, but viewing $A\left(Y\right)$ as a NSF over $\overline{Y}$.
\end{proof}

\begin{remark}\label{vna1}
In the specific case when $X=Y$ and $\omega_{i}$ are the coordinate functions, we get that $\Ho_{\8}\left(X\right)\subset Mult\left(\mathbf{F}\right)$.
\qed\end{remark}

Using von Neumann's or Ando's inequality (see \cite[Theorem 1.2]{pisier}), or it's weaker version for more than $2$ variables (see \cite{lubin}) we can get the following corollary.

\begin{corollary}\label{ando}
Let $\mathbf{H}$ be a $1$-independent HSF over a set $X$ and let $\omega_{1},...,\omega_{n}\in \Bo_{Mult\left(\mathbf{H}\right)}$. If $n=1$ or $n=2$, then $C_{\vec{\omega}}$ is a contraction from $A\left(\D^{n}\right)$ into $Mult\left(\mathbf{H}\right)$. If $n>2$, then $C_{\vec{\omega}}$ is a contraction from $A\left(\sqrt{n}\D^{n}\right)$ into $Mult\left(\mathbf{H}\right)$.
\end{corollary}

It was proven in \cite[Theorem 2.5]{erz3} that a non-constant multiplier on a weakly compactly embedded NSCF over a connected space does not attain it's multiplier norm. This leads to the following refinement of  Corollary \ref{ando}.

\begin{corollary}\label{ando2}
Let $\mathbf{H}$ be a $1$-independent HSCF over a connected topological space $X$ and let $\omega_{1},...,\omega_{n}\in \Bo_{Mult\left(\mathbf{H}\right)}$ be non-constants. If $n=1$ or $n=2$, then $C_{\vec{\omega}}$ is a contraction from $\Ho_{\8}\left(\D^{n}\right)$ into $Mult\left(\mathbf{H}\right)$. If $n>2$, then $C_{\vec{\omega}}$ is a contraction from $\Ho_{\8}\left(\sqrt{n}\D^{n}\right)$ into $Mult\left(\mathbf{H}\right)$.
\end{corollary}

\begin{remark}
A direct way to prove the result is using the functional calculus for contractions on Hilbert spaces (see \cite[V.4]{cooper} and \cite{bro}). Moreover, the single variable calculus also implies that $C_{\omega}$ is continuous with respect to the topology of convergence almost everywhere on $\partial\D$ in $\Ho^{\8}$ and the strong operator topology on $Mult\left(\mathbf{H}\right)$.
\qed\end{remark}

In a way, we gave a sufficient condition for a multiplier algebra to be large. Unsurprisingly, it can be modified to get a sufficient condition for non-separability.

\begin{theorem}\label{sep}
Let $\mathbf{H}$ be a $1$-independent HSCF over a connected topological space $X$ and let $\omega\in Mult\left(\mathbf{H}\right)$ be non-constant and such that $\|\omega\|_{Mult\left(\mathbf{H}\right)}=\|\omega\|_{\8}$. Then $Mult\left(\mathbf{H}\right)$ is not separable.
\end{theorem}

Note that a reflexive NSCF over a separable topological space is separable, and so if under the assumptions of the theorem $X$ is separable, $Mult\left(\mathbf{H}\right)$ is not reflexive.

Before proving the theorem, let us consider an auxiliary object. If $Y\subset \D$ has a limit point in $\D$, then $\Ho^{\8}_{Y}=\left(\Ho^{\8}, \|\cdot\|_{\8}^{Y}\right)$ is a normed space. For a discussion on when it is a NSHF or a Banach space see \cite{ab} (see also \cite{bv}), while here we will focus on separability. Namely, let us prove the following result.

\begin{lemma}\label{sep1}
$\Ho^{\8}_{Y}$ is separable if and only if $\overline{Y}\subset \D$.
\end{lemma}
\begin{proof}
Sufficiency: If $\overline{Y}\subset \D$, then $Y\subset D$, where $D\subset\D$ is a smaller disk. As was mentioned in Example \ref{ad}, the polynomials are dense in $A\left(D\right)$, and since $\Ho^{\8}\subset A\left(D\right)$ with $\|\cdot\|_{\8}^{Y}\le \|\cdot\|_{\8}^{D}$, the polynomials are dense in $\Ho^{\8}_{Y}$. Hence, the latter is separable.\medskip

Necessity: If $\overline{Y}\cap\partial\D\ne\varnothing$ inductively construct the following sequence of elements of $Y$. Take any $y_{1}\in Y$ and assume that $y_{1},...,y_{n}$ are already chosen. Since there are elements in $Y$ arbitrarily close to $\partial\D$ there is $y_{n+1}\in Y$ such that $1-\left|y_{y+1}\right|\le \frac{1}{2}\left(1-\left|y_{n}\right|\right)$. Then, from the Carleson's interpolation theorem (combine \cite[theorems 9.1 and 9.2]{duren}) the restriction operator $\Ho^{\8}\ni f\to \left\{f\left(y_{n}\right)\right\}_{n\in\N}\in l^{\8}$ is a bijection. Hence, for any $N\subset \N$ there is $f_{N}\in \Ho^{\8}$ such that $f_{N}\left(y_{n}\right)=1$, when $n\in N$, and $f_{N}\left(y_{n}\right)=0$, otherwise. Therefore, the set $\left\{f_{N},~N\subset\N\right\}$ has cardinality continuum, and $\|f_{N}-f_{M}\|_{\8}^{Y}\ge 1$, for any distinct $M,N\subset \N$. Thus, $\Ho^{\8}_{Y}$ cannot be separable.
\end{proof}

\begin{proof}[Proof of Theorem \ref{sep}]
Without loss of generality we may assume that $\|\omega\|_{Mult\left(\mathbf{H}\right)}=1=\|\omega\|_{\8}$, and since $\omega$ is not a constant, $Y=\omega\left(X\right)\subset \D$. Since $\|\omega\|_{\8}=1$ it follows that $\overline{Y}\cap \partial\D\ne\varnothing$, and so from the preceding lemma there is a collection $\left\{g_{t}\right\}_{t\in\R}\subset \Ho^{\8}$ such that $\|g_{t}-g_{s}\|_{\8}^{A}\ge 1$, for every distinct $t,s\in\R$. From Corollary \ref{ando2}, $C_{\omega}\Ho^{\8}\subset Mult\left(\mathbf{H}\right)$, and so $g_{t}\circ\omega\in Mult\left(\mathbf{H}\right)$, for every $t\in\R$. Moreover, $$\|g_{t}\circ\omega-g_{s}\circ\omega\|_{ Mult\left(\mathbf{H}\right)}\ge\|g_{t}\circ\omega-g_{s}\circ\omega\|_{\8}^{X}=\|g_{t}-g_{s}\|_{\8}^{Y}\ge 1,$$ for distinct $t,s\in\R$, from where $Mult\left(\mathbf{H}\right)$ is non-separable.
\end{proof}

Note that the assumption of the theorem do not imply that $Mult\left(\mathbf{H}\right)$ is  a closed subalgebra of $\Ho_{\8}\left(X\right)$.

\begin{example}
Let $\mathbf{H}$ be a HSHF over $X=\B_{\C^{2}}$ whose reproducing kernel is $K\left(x,y\right)=\frac{1}{1-x_{1}\overline{y_{1}}-x_{2}\overline{y_{2}}}$, where $x=\left(x_{1},x_{2}\right)$ and $y=\left(y_{1},y_{2}\right)$. Let $\omega_{i}:X\to \D$ be defined by $\omega_{i}\left(x_{1},x_{2}\right)=x_{i}$, $i=1,2$. For $x,y\in X$ we have $$\left(1-\omega_{1}\left(x\right)\overline{\omega_{1}\left(y\right)}\right)K\left(x,y\right)=\frac{1-x_{1}\overline{y_{1}}}{1-x_{1}\overline{y_{1}}-x_{2}\overline{y_{2}}}=\frac{1}{1-\frac{x_{2}\overline{y_{2}}}{1-x_{1}\overline{y_{1}}}}=\frac{1}{1-L\left(x,y\right)},$$
where $L=\frac{\omega_{2}\otimes \overline{\omega_{2}}}{1-\omega_{1}\otimes \overline{\omega_{1}}}$ is a kernel with $\left|L\right|<1$ (see Example \ref{rkhs}). Hence, $\left(1-\omega_{1}\otimes\overline{\omega_{1}}\right)K$ is a kernel, and so $\|\omega_{1}\|_{Mult\left(\mathbf{H}\right)}=1$. Analogously, $\|\omega_{2}\|_{Mult\left(\mathbf{H}\right)}=1$. On the other hand, if $Mult\left(\mathbf{H}\right)$ is isomorphic to a subalgebra of $\Ho_{\8}\left(X\right)$, then since it contains $\omega_{i}$, according to Remark \ref{vna1}, $\Ho_{\8}\left(X\right)=Mult\left(\mathbf{H}\right)$. However, the multiplier algebra of $\mathbf{H}$ is not $\Ho_{\8}\left(X\right)$ (see \cite[Remark 8.9]{am}). Contradiction.
\qed\end{example}

The theorem begs the following question:

\begin{question}
Let $\mathbf{H}$ be a $1$-independent HSCF over a connected topological space $X$ such that $Mult\left(\mathbf{H}\right)$ is not separable. Is there always a non-constant  $\omega\in Mult\left(\mathbf{H}\right)$ such that $\|\omega\|_{Mult\left(\mathbf{H}\right)}=\|\omega\|_{\8}$?
\end{question}

\section{Acknowledgment}

The author wants to thank Yemon Choi who contributed Example \ref{ardy}, Jochen Wengenroth and Giorgio Metafune, who contributed to the proof of Lemma \ref{sep1}, and the service \href{mathoverflow.com/}{MathOverflow} which made it possible.

\end{document}